\documentclass[11pt]{article}

\usepackage[draft]{ed}
\usepackage{amsfonts}
\usepackage{amsmath}
\usepackage{amssymb}
\usepackage{epsfig}
\usepackage{bm}
\usepackage{enumerate}
\usepackage{multicol}
\usepackage{multirow}
\numberwithin{equation}{section} 

\usepackage{color}
\usepackage[normalem]{ulem}

\title{Well-conditioned boundary integral equation formulations and Nystr\"om discretizations for the solution of Helmholtz problems with impedance boundary conditions in two-dimensional Lipschitz domains}
\author{ Catalin Turc\thanks{  Department of
Mathematical Sciences and Center for Applied Mathematics and Statistics, New Jersey  Institute of Technology,
Univ. Heights. 323 Dr. M. L. King Jr. Blvd, Newark, NJ 07102, USA, e-mail: catalin.c.turc@njit.edu.}, Yassine Boubendir\thanks{  Department of
Mathematical Sciences and Center for Applied Mathematics and Statistics, New Jersey  Institute of Technology,
Univ. Heights. 323 Dr. M. L. King Jr. Blvd, Newark, NJ 07102, USA, e-mail: boubendi@njit.edu.}, Mohamed Kamel Riahi\thanks{  Department of
Mathematical Sciences and Center for Applied Mathematics and Statistics, New Jersey  Institute of Technology,
Univ. Heights. 323 Dr. M. L. King Jr. Blvd, Newark, NJ 07102, USA, e-mail: riahi@njit.edu.}}

\newtheorem{theorem}{Theorem}[section]

\newtheorem{remark}[theorem]{Remark}
\newenvironment{proof}{\hspace{0.5cm} {\bf Proof.}}
{$\quad {}_\blacksquare$\vspace{0.3cm}}

\date{}
\newcommand{\triple}[1]{{\left\vert\kern-0.25ex\left\vert\kern-0.25ex\left\vert #1 
    \right\vert\kern-0.25ex\right\vert\kern-0.25ex\right\vert}}

\begin{document}
\maketitle
\begin{abstract}
  We present a regularization strategy that leads to well-conditioned boundary integral equation formulations of Helmholtz equations with impedance boundary conditions in two-dimensional Lipschitz domains. We consider both the case of classical impedance boundary conditions, as well as the case of transmission impedance conditions wherein the impedances are certain coercive operators. The latter type of problems is instrumental in the speed up of the convergence of Domain Decomposition Methods for Helmholtz problems. Our regularized formulations use as unknowns the Dirichlet traces of the solution on the boundary of the domain. Taking advantage of the increased regularity of the unknowns in our formulations, we show through a variety of numerical results that a graded-mesh based Nystr\"om discretization of these regularized formulations leads to efficient and accurate solutions of interior and exterior Helmholtz problems with impedance boundary conditions.   
   
\end{abstract}

\section{Introduction\label{intro}}

\parskip 2pt plus2pt minus1pt

The computation of accurate solutions of Helmholtz problems with impedance boundary conditions is relevant to a wide variety of applications, including antennas and stealth technology. Another important area where numerical solutions of impedance boundary value problems are extremely relevant is that of Domain Decomposition Methods (DDM) for the solution of Helmholtz equations. Indeed, in the aforementioned context DDM rely on impedance matching boundary conditions between subdomain solutions~\cite{Depres}. In order to accelerate the convergence of DDM for Helmholtz equations, impedance (Robin) transmission conditions can be used to great effect~\cite{boubendirDDM,Steinbach} on the interfaces between subdomains. In these cases the impedance (which is typically a piecewise constant function) on the interface between two subdomains  is replaced by certain coercive operators that are approximations to Dirichlet to Neumann operators corresponding to those subdomains~\cite{boubendirDDM}.

Whenever applicable, boundary integral solvers for solution of Helmholtz impedance boundary value problems are computationally advantageous~\cite{Langdon,Bernard,Perrey-Debain}. Although both interior and exterior Helmholtz impedance boundary value problems remain well-posed for all real values of the frequency, robust boundary integral formulations of these problems still have to rely on the Combined Field approach~\cite{KressColton}. The classical Combined Field formulations feature the Helmholtz hypersingular boundary integral operator, and as such are not integral equations of the second kind. We present in this paper regularized combined field integral equations of the second kind for Helmholtz impedance boundary value problems in two dimensional Lipschitz domains. Our regularization strategy was previously applied successfully to Neumann boundary conditions~\cite{turc_corner_N,turc1,br-turc}.  Our approach covers both the cases of piecewise constant impedance, as well as the transmission impedance operators of importance to DDM. The unknowns in our regularized formulations are Dirichlet traces of solutions on the boundary, which enjoy optimal regularity properties amongst solutions of possible boundary integral formulations of Helmholtz impedance problems in Lipschitz domains. 

We take advantage of the increased regularity of the solutions of our regularized formulations (the solutions are H\"older continuous) to construct high-order Nystr\"om discretizations based on graded meshes, trigonometric interpolation, singular kernel-splitting, and analytic evaluations of integrals that involve products of certain singular functions and Fourier harmonics~\cite{kusmaul,martensen}. Our Nystr\"om method incorporates sigmoid transforms~\cite{KressCorner} within parametrizations of domains with corners and it uses the Jacobians of these transformations as multiplicative weights to define new unknowns. A weighted Dirichlet trace defined as the product of the derivatives of the sigmoid parametrizations and the usual Dirichlet trace of solution of impedance problems is introduced as a new unknown. Given that the derivatives of the parametrizations that incorporate sigmoid transforms vanish polynomially at corners, the weighted traces are more regular for large enough values of the order of the polynomial in the sigmoid transform. Introducing new weighted unknowns also require definition of new weighted boundary integral equations that involve weighted versions of the four scattering boundary integral operators. The weighted formulations turn out to be particularly useful in the case of piecewise constant (discontinuous) impedances. We use splitting of the kernels of the four Helmholtz boundary integral operators required in the Calder\'on calculus into regular components and explicit singular components that have been presented in our previous efforts~\cite{turc_corner_N,dominguez2015well}. An appealing aspect of our regularized formulations is exploiting Calder\'on's identities to bypass evaluations of hypersingular operators, which facilitate the kernel splitting techniques. We give ample numerical evidence that our Nystr\"om solvers for impedance boundary value problems converge with high-order and are well-conditioned throughout the frequency spectrum.

The paper is organized as follows: in Section~\ref{cfie} we formulate the Helmholtz impedance boundary value problems we are interested in; in Section~\ref{di_ind_cfie} we discuss several regularized boundary integral formulations of the Helmholtz impedance boundary value problems and we establish the well-posedness of these regularized formulations; in Section~\ref{transmission} we investigate regularized boundary integral formulations for transmission impedance boundary value problems in connection with Domain Decomposition Methods; finally, in Section~\ref{singular_int} we present high-order Nystr\"om discretizations of the various boundary integral equations considered in this paper.

\section{Integral Equations of Helmholtz impedance boundary value problems\label{cfie}}

We consider the problem of evaluating time-harmonic fields that satisfy impedance boundary conditions on the boundary $\Gamma$ of a scatterer $D_2$ which occupies a bounded region in $\mathbb{R}^2$. Denoting by $D_1=\mathbb{R}^2\setminus {\overline{D_2}}$, we are interesting in solving
\begin{equation} 
  \label{eq:Ac_i}
\begin{aligned}
  \Delta u^j+k^2 u^j&=&0,\qquad &\mathrm{in}\ D^j,\ j=1,2\\
\gamma_N^j u^j+ Z^j \gamma_D^j u^j&=&f^j,\qquad &\mathrm{on}\ \Gamma,\ j=1,2,
\end{aligned}
\end{equation}
where the wavenumber $k$ is assumed to be positive, $f^j$ are data defined on the curve $\Gamma$, and $Z^j\in\mathbb{C}$ such that $\Im{Z^1}>0$ and $\pm \Im{Z^2}> 0$. In equations~\eqref{eq:Ac_i} and what follows $\gamma_D^j,j=1,2$ denote exterior and respectively interior Dirichlet traces, whereas $\gamma_N^j,j=1,2$ denote exterior and respectively interior Neumann traces taken with respect to the exterior unit normal on $\Gamma$. We assume in what follows that the boundary $\Gamma$ is a closed Lipschitz curve in $\mathbb{R}^2$.  

For any $D\subset\mathbb{R}^2$ domain with bounded Lipschitz boundary $\Gamma$,  we denote by $H^s(D)$ the classical Sobolev space of order $s$ on $D$ 
(see for example~\cite[Ch. 3]{mclean:2000} or~\cite[Ch. 2]{adams:2003}). 
We consider in addition the Sobolev spaces defined on the boundary $\Gamma$,  $H^s(\Gamma)$, which are well defined for any $s\in[-1,1]$. We recall that for any $s>t$, $H^s(\Sigma)\subset H^t(\Sigma)$, $\Sigma\in\{D_1,D_2,\Gamma\}$ and the embeddings are compact. Moreover, 
$\big(H^t(\Gamma)\big)'=H^{-t}(\Gamma)$ when the inner product of $H^0(\Gamma)=L^2(\Gamma)$ is used as duality product. If $\Gamma_0\subset\Gamma$ such that $meas(\Gamma_0)>0$ (we mean here the one dimensional measure), we can still define Sobolev spaces of functions/distributions on $\Gamma_0$. Indeed, for $0<s\leq 1/2$ we define by ${H}^s(\Gamma_0)$ be the space of distributions that are restrictions to $\Gamma_0$ of functions in $H^s(\Gamma)$. The space $\widetilde{H}^s(\Gamma_0)$ is defined as the closed subspace of $H^s(\Gamma_0)$
\[
\widetilde{H}^s(\Gamma_0)=\{u\in H^s(\Gamma_0):\widetilde{u}\in H^s(\Gamma)\},\ 0<s\leq 1/2
\]
where
\[\widetilde{u}:=\begin{cases}
 u, & {\rm on}\  \Gamma \\
 0,  & {\rm on}\ \Gamma\setminus\Gamma_0.
\end{cases}
\]
We define then $H^t(\Gamma_0)$ to be the dual of $\widetilde{H}^{-t}(\Gamma_0)$ for $-1/2\leq t<0$, and $\widetilde{H}^t(\Gamma_0)$ the dual of $H^{-t}(\Gamma_0)$ for $-1/2\leq t<0$.

It is well known that $\gamma_D^j: H^{s+1/2}(D_j)\to H^s(\Gamma)$ is continuous for $s\in(0,1)$ and if
\[
 H^s_\Delta(D_j):=\left\{U\in H^{s}(D_j)\ :\ \Delta U\in L^2(D_j) \right\},
\]
endowed with its natural norm, 
then $\gamma_N:  H^s_\Delta(D_j)\to H^{s-3/2}(\Gamma)$ is continuous  for  $s\in(1/2,3/2)$. The space $H^1(\Gamma)$, and its dual $H^{-1}(\Gamma)$, are then the limit case from several different perspectives. 

If we furthermore require that $u^1$ satisfies Sommerfeld radiation conditions at infinity:
\begin{equation}\label{eq:radiation}
\lim_{|r|\to\infty}r^{1/2}(\partial u^1/\partial r - iku^1)=0,
\end{equation}
then the assumptions $\Im{Z^1}>0$ and $\pm\Im{Z^2}>0$ guarantee that equations~\eqref{eq:Ac_i} have unique solutions $u^1\in C^2(D_1)\cap H^1_{\rm loc}(D_1)$ and $u^2\in C^2(D_2)\cap H^1(D_2)$ for data $f^j\in H^{-1/2}(\Gamma)$~\cite{mclean:2000}. The unique solvability results remain valid in the cases when $Z^1\in L^\infty(\Gamma),\ \Im{Z^1}>0$ and $Z^2\in L^\infty(\Gamma),\ \pm\Im(Z^2)>0$~\cite{mclean:2000}.

We note that in many applications of interest the data $f^1$ is related to an incident field $u^{inc}$ that satisfies 
\begin{equation}
  \label{eq:Maxwell_inc}
  \Delta u^{inc}+k^2 u^{inc}=0 \qquad \mathrm{in}\ \overline{D}_1,
\end{equation}
by the relation
\begin{equation}\label{eq:rhs}
f^1 = -\gamma_N^1 u^{inc} - Z^1 \gamma_D^1 u^{inc},
\end{equation}
in which case the solution $u^1$ of equations~\eqref{eq:Ac_i} is a scattered field.

\section{Regularized boundary integral formulations for the solution of Helmholtz impedance boundary value problems\label{di_ind_cfie}}

We present next regularized direct boundary integral formulations for the solution of impedance boundary value problems that are similar in spirit to those introduced in~\cite{turc1,turc_corner_N} in the case of Neumann boundary conditions. To this end, we begin by reviewing the definition and mapping properties of the four scattering boundary integral operators related to the Helmholtz operator $\Delta +k^2$. 

\subsection{Layer potentials and operators}
We start with the definition of the single and double layer potentials. Given a wavenumber $k$ such that $\Re{k}>0$ and $\Im{k}\geq 0$, and a density $\varphi$ defined on $\Gamma$, we define the single layer potential as
$$[SL_k(\varphi)](\mathbf{z}):=\int_\Gamma G_k(\mathbf{z}-\mathbf{y})\varphi(\mathbf{y})ds(\mathbf{y}),\ \mathbf{z}\in\mathbb{R}^2\setminus\Gamma$$
and the double layer potential as
$$[DL_k(\varphi)](\mathbf{z}):=\int_\Gamma \frac{\partial G_k(\mathbf{z}-\mathbf{y})}{\partial\mathbf{n}(\mathbf{y})}\varphi(\mathbf{y})ds(\mathbf{y}),\ \mathbf{z}\in\mathbb{R}^2\setminus\Gamma$$
where $G_k(\mathbf{x})=\frac{i}{4}H_0^{(1)}(k|\mathbf{x}|)$ represents the two-dimensional outgoing Green's function of the Helmholtz equation with wavenumber $k$. The Dirichlet and Neumann exterior and interior traces on $\Gamma$ of the single and double layer potentials corresponding to the wavenumber $k$ and a density $\varphi$ are given by
\begin{eqnarray}\label{traces}
\gamma_D^1 SL_k(\varphi)&=&\gamma_D^2 SL_k(\varphi)=S_k\varphi \nonumber\\
\gamma_N^j SL_k(\varphi)&=&(-1)^j\frac{\varphi}{2}+K_k^\top \varphi\quad j=1,2\nonumber\\
\gamma_D^j DL_k(\varphi)&=&(-1)^{j+1}\frac{\varphi}{2}+K_k\varphi\quad j=1,2\nonumber\\
\gamma_N^1 DL_k(\varphi)&=&\gamma_N^2 DL_k(\varphi)=N_k\varphi.
\end{eqnarray}
In equations~\eqref{traces} the operators $K_k$ and $K^\top_k$, usually
referred to as double and adjoint double layer operators, are defined for a given wavenumber $k$ and density $\varphi$ as
\begin{equation}
\label{eq:double}
(K_k\varphi)(\mathbf x):=\int_{\Gamma}\frac{\partial G_k(\mathbf x-\mathbf y)}{\partial\mathbf{n}(\mathbf y)}\varphi(\mathbf y)ds(\mathbf y),\ \mathbf x\ \in \Gamma
\end{equation}
and 
\begin{equation}
\label{eq:adj_double}
(K_k^\top\varphi)(\mathbf x):=\int_{\Gamma}\frac{\partial G_k(\mathbf x-\mathbf y)}{\partial\mathbf{n}(\mathbf x)}\varphi(\mathbf y)ds(\mathbf y),\ \mathbf x \in \Gamma.
\end{equation}
Furthermore, for a given wavenumber $k$ and density $\varphi$, 
the operator $N_k$ denotes the Neumann trace of the double layer potential on 
$\Gamma$ given in terms of a Hadamard Finite Part (FP)  integral which can be re-expressed in terms of a Cauchy Principal Value (PV) integral that involves the tangential derivative $\partial_s$ on the curve $\Gamma$
\begin{eqnarray}
 \label{eq:normal_double} 
(N_k \varphi)(\mathbf x) &:=& \text{FP} \int_\Gamma \frac{\partial^{2}G_k(\mathbf x -\mathbf y)}{\partial \mathbf{n}(\mathbf x) \partial \mathbf{n}(\mathbf y)} \varphi(\mathbf y)ds(\mathbf y)\nonumber\\
&=&k^{2}\int_\Gamma G_k(\mathbf x -\mathbf y)
(\mathbf{n}(\mathbf x)\cdot\mathbf{n}(\mathbf y))\varphi(\mathbf y)ds(\mathbf y)+ {\rm PV}
\int_\Gamma \partial_s G_k(\mathbf x -\mathbf y)\partial_s \varphi(\mathbf y)ds(\mathbf y).
\end{eqnarray}
Finally, the single layer operator $S_k$ is defined for a wavenumber $k$ as
\begin{equation}\label{eq:sl}
(S_k\varphi)(\mathbf x):=\int_\Gamma G_k(\mathbf x -\mathbf y)\varphi(\mathbf y)ds(\mathbf y),\ \mathbf{x}\ \in \Gamma
\end{equation} 
for a density function $\varphi$ defined on $\Gamma$. 

 Green identities can be now written in the simple form:
\[
 u^j=(-1)^j SL_k(\gamma_N^j u^j)-(-1)^j DL_k(\gamma_D^j u^j). 
\]
Similarly, 
\begin{equation}\label{eq:C1}
 C_j=\tfrac12\begin{bmatrix}I\\ &I\end{bmatrix}+(-1)^j\begin{bmatrix}
                           -K_k & S_k\\
                           -N_k& K_k^\top
                          \end{bmatrix}, \quad 
                          j=1,2
\end{equation}
are the Calder\'{o}n exterior/interior projections associated to the exterior/interior Helmholtz equation: 
\begin{equation}\label{eq:C2}
 C_j^2=C_j,\quad C_j\begin{bmatrix}
                     \gamma_D^j u^j\\
                     \gamma_N^j u^j
                    \end{bmatrix}=
\begin{bmatrix}
                     \gamma_D^j u^j\\
                     \gamma_N^j u^j
                    \end{bmatrix}.
\end{equation}
We recall that from \eqref{eq:C1}-\eqref{eq:C2} one deduces easily 
\begin{equation}\label{eq:calderon}
 S_kN_k=-\frac{1}{4}I + K_k^2,\quad 
 N_kS_k=-\frac{1}{4}I + (K_k^\top)^2,\quad 
 N_k K_k= K_k^\top N_k.
\end{equation}

We recount next several important results related to the mapping properties of the four boundary integral operators of the Calder\'{o}n calculus~\cite{dominguez2015well}.

\begin{theorem}\label{mapping}
  Let  $D_2$ be a bounded domain, with Lipschitz boundary $\Gamma$. The following mappings
\begin{itemize}
\item $S_k:H^{s}(\Gamma)\to H^{s+1}(\Gamma)$
\item $K_k:H^{s+1}(\Gamma)\to H^{s+1}(\Gamma)$
\item $K^\top_k:H^{s}(\Gamma)\to H^{s}(\Gamma)$
\item $N_k:H^{s+1}(\Gamma)\to H^{s}(\Gamma)$
\end{itemize}
are continuous for $s\in[-1,0]$. Furthermore, if $k_1\ne k_2$ we have that  
\begin{itemize}
\item $S_{k_1}-S_{k_2}:H^{-1}(\Gamma)\to H^{1}(\Gamma)$
\item $K_{k_1}-K_{k_2}:H^{0}(\Gamma)\to H^{1}(\Gamma)$
\item $K^\top_{k_1}-K^\top_{k_2}:H^{-1}(\Gamma)\to H^{0}(\Gamma)$
\item $N_{k_1}-N_{k_2}:H^{0}(\Gamma)\to H^{0}(\Gamma)$.
\end{itemize}
are continuous and compact. 
\end{theorem}

We also recount a result due to Escauriaza, Fabes and Verchota~\cite{EsFaVer:1992}. In this result, $K_0$, $K_0^\top$ are the double and adjoint double layer operator for Laplace equation (which obviously correspond to $k=0$).

\begin{theorem}\label{theo:inv}
For any Lipschitz curve $\Gamma$ and  $\lambda\not\in [-1/2,1/2)$, the mappings
\[
 \lambda I+K_0 :H^s(\Gamma)\to H^s(\Gamma)
\]
are invertible for $s\in[-1,1]$. Furthermore, the mappings
\[
\frac{1}{2}I\pm K_0:H^s(\Gamma)\to H^s(\Gamma)
\]
are Fredholm of index 0 for $s\in[-1,1]$.
\end{theorem}

\subsection{Regularized boundary integral equation formulations of Helmholtz impedance boundary value problems}

We start with the case of exterior scattering problems with impedance boundary conditions given by~\eqref{eq:rhs} and we derive {\em direct} regularized boundary integral equations formulations of these problems. Assuming smooth incident fields $u^{inc}$ in $\mathbb{R}^2$, an application of the second Green identities for the functions $u^{inc}$ and $G_k(\mathbf{x}-\cdot),\ \mathbf{x}\in D_1$ in the domain $D_2$ leads to
\[
0=-SL_k(\gamma_N^1 u^{inc})+DL_k(\gamma_D^1 u^{inc})\qquad {\rm in}\qquad D_1
\]
and hence
\[
u^1 = -SL_k[\gamma_N^1(u^1 +u^{inc})] +DL_k[\gamma_D^1(u^1 +u^{inc})]\qquad {\rm in}\qquad D_1.
\]
We define the {\em physical} unknown that is the Dirichlet trace of the total field on $\Gamma$
\begin{equation}\label{unknown1}
\gamma_D^1 u:=\gamma_D^1(u^1 +u^{inc})
\end{equation}
and take into account the impedance boundary conditions to get the representation formula
\begin{equation}\label{eq:repr1}
u^1 = SL_k(Z^1 \gamma_D^1u) + DL_k(\gamma_D^1 u).
\end{equation}
Applying the exterior Dirichlet and Neumann traces to equation~\eqref{eq:repr1} we obtain
\begin{eqnarray}\label{eq:traces1}
\frac{\gamma_D^1 u}{2} - K_k(\gamma_D^1 u) - S_k(Z^1\gamma_D^1u)&=&\gamma_D^1 u^{inc}\nonumber\\
\frac{Z^1\gamma_D^1 u}{2}+N_k(\gamma_D^1 u)+K_k^\top(Z^1\gamma_D^1 u)&=&-\gamma_N^1 u^{inc}.\nonumber
\end{eqnarray}
Following the strategy introduced in~\cite{turc_corner_N} we add the first equation above to the second equation above composed on the left with the operator $-2S_{\kappa},\ \Im{\kappa}>0$ and we obtain a Regularized Combined Field Integral Equation (CFIER) of the form
\begin{eqnarray}\label{eq:CFIER1}
\mathcal{A}_{k,\kappa}^1 \gamma_D^1 u&=&\gamma_D^1 u^{inc}+2S_\kappa\gamma_N^1 u^{inc}\nonumber\\
\mathcal{A}_{k,\kappa}^1&:=&\frac{1}{2}I-2S_\kappa N_k-S_\kappa Z^1-2S_\kappa K_k^\top Z^1-K_k-S_k Z^1. 
\end{eqnarray} 
\begin{remark}
For the time being we view $Z^1$ as the multiplicative operator by the complex constant $Z^1$. The notation in equation~\eqref{eq:CFIER1} allows us to consider more general operators $Z^1$.
\end{remark}
Similar considerations lead us to regularized boundary integral equation formulations of interior Helmholtz impedance boundary value problems. Indeed, the physical unknown $\gamma_D^2 u^2$ satisfies
\begin{eqnarray}\label{eq:CFIER2}
\mathcal{A}_{k,\kappa}^2 \gamma_D^2 u^2&=&(S_k+S_\kappa-2S_\kappa K_k^\top)f^2\nonumber\\
\mathcal{A}_{k,\kappa}^2&:=&\frac{1}{2}I-2S_\kappa N_k+S_\kappa Z^2-2S_\kappa K_k^\top Z^2+K_k+S_k Z^2. 
\end{eqnarray} 
We will establish the well-posedness of the CFIER formulations in appropriate Sobolev spaces. Although for the time being we assume that $Z^j,\ j=1,2$ are complex constants, the derivations we present next remain valid for the cases when $Z^j,\ j=1,2$ are functions defined on $\Gamma$. We note that in the case $Z^j\in L^\infty(\Gamma),\ j=1,2$, we have that $\gamma_D^j u^j\in H^{1/2}(\Gamma)$~\cite{mclean:2000} and hence $Z^j\gamma_D^j u^j\in L^2(\Gamma)$. Assuming impedance boundary data $f^j\in L^2(\Gamma)$, it follows that $\gamma_N^j u^j\in L^2(\Gamma)$, which in turn imply $\gamma_D^j u^j\in H^1(\Gamma)$. In the light of this discussion, we will establish the well-posedness of the CFIER equations~\eqref{eq:CFIER1} and~\eqref{eq:CFIER2} respectively in a wide range of Sobolev spaces.

\subsection{Well-posedness of the CFIER formulations~\eqref{eq:CFIER1} and~\eqref{eq:CFIER2}}

We make use of the classical results recounted in Theorem~\ref{mapping} and Theorem~\ref{theo:inv} to establish the following result:
\begin{theorem}\label{thm1}
Assume that $Z^1\in\mathbb{C}$ such that $\Im{Z^1}>0$. The operators $\mathcal{A}_{k,\kappa}^1$ defined in equations~\eqref{eq:CFIER1} are invertible with continuous inverses in the spaces $H^s(\Gamma)$ for all $s\in[-1,1]$.
\end{theorem}
\begin{proof}
We establish first that the operators $\mathcal{A}_{k,\kappa}^1$  are Fredholm of index 0 in $H^0(\Gamma)$. Using Calder\'on's identities we can recast $\mathcal{A}_{k,\kappa}^1$ into the following form
\begin{eqnarray}
\mathcal{A}_{k,\kappa}^1&=&(I-K_0-2K_0^2) + \mathcal{A}_0^1=2\left(\frac{1}{2}I-K_0\right)\left(I+K_0\right)+ \mathcal{A}_0^1\nonumber\\
\mathcal{A}_0^1&=&2S_\kappa(N_\kappa-N_k)-2(K_\kappa-K_0)K_\kappa-2K_0(K_\kappa-K_0)-S_\kappa Z^1-2S_\kappa K_k^\top Z^1\nonumber\\
& +&(K_0-K_k)-S_k Z^1.\nonumber
\end{eqnarray}
It follows from the results in Theorem~\ref{mapping} that $\mathcal{A}_0^1:H^0(\Gamma)\to H^1(\Gamma)$ continuously, and thus $\mathcal{A}_0^1:H^0(\Gamma)\to H^0(\Gamma)$ is compact. Also, the operator 
\[
2\left(\frac{1}{2}I-K_0\right)\left(I+K_0\right)
\]
is Fredholm of index 0 in $H^0(\Gamma)$ since (a) the operator $\frac{1}{2}I-K_0$ is Fredholm of index 0 in $H^0(\Gamma)$, (b) the operator $I+K_0$ is invertible in $H^0(\Gamma)$, and (c) the two operators commute. We thus conclude that the operator $\mathcal{A}_{k,\kappa}^1$ is a compact perturbation of a Fredholm operator of index 0 in the space $H^0(\Gamma)$, and hence the operator $\mathcal{A}_{k,\kappa}^1$ is itself a Fredholm operator of index 0 in the same space.

Given the Fredholm property of the operator $\mathcal{A}_{k,\kappa}^1$, its invertibility is equivalent to its injectivity. We show in turn that the transpose of this operator with respect to the real scalar product in $H^0(\Gamma)$ is injective. The latter can be seen to equal
\[
(\mathcal{A}_{k,\kappa}^1)^\top=\frac{1}{2}I-2N_k S_\kappa-Z^1 S_\kappa- 2 Z^1 K_k S_\kappa-K_k^\top-Z^1 S_k.
\]
Let $\varphi\in Ker((\mathcal{A}_{k,\kappa}^1)^\top)$ and let us define
\[
v:=SL_k\varphi+DL_k[2S_\kappa]\varphi,\qquad \mathrm{in}\quad \mathbb{R}^2\setminus\Gamma.
\]
We have that
\begin{eqnarray}
\gamma_D^1 v &=& S_\kappa\varphi + 2K_k S_\kappa \varphi+S_k \varphi\nonumber\\
\gamma_N^1 v &=&-\frac{1}{2}\varphi + K_k^\top\varphi + 2N_k S_\kappa\varphi\nonumber
\end{eqnarray}
and hence
\[
\gamma_N^1 v + Z^1 \gamma_D^1 v = 0
\]
if we take into account that $\varphi\in Ker((\mathcal{A}_{k,\kappa}^1)^\top)$. Now $v$ is a radiative solution of Helmholtz equation in $D_1$ satisfying the impedance boundary condition $\gamma_N^1 v + Z^1 \gamma_D^1 v = 0$. Under the assumption that $\Im{Z^1}>0$ it follows that $v$ is identically zero in $D_1$, and hence 
\[
\gamma_D^1 v = 0\qquad \gamma_N^1 v =0.
\]
The last relation immediately implies
\[
\gamma_D^2 v = -2S_\kappa\varphi\qquad \gamma_N^2 v = \varphi.
\]
Using Green's formulas we obtain that
\[
\int_{D_2}(|\nabla v|^2 -k |v|^2)dx=-2\int_\Gamma (S_\kappa\varphi)\ \overline{\varphi}\ ds.
\]
Using the fact that~\cite{turc2} 
\[
\Im \int_\Gamma (S_\kappa\varphi)\ \overline{\varphi}\ ds >0,\quad \varphi\neq 0
\]
when $\Im{\kappa}>0$ we obtain that $\varphi=0$. Consequently, the operator $(\mathcal{A}_{k,\kappa}^1)^\top$ is injective, and thus the operator $\mathcal{A}_{k,\kappa}^1$ is injective as well, which completes the proof of the Theorem in the space $H^0(\Gamma)$. Clearly, the arguments of the proof can be repeated verbatim in the Sobolev spaces $H^s(\Gamma)$ for all $s\in[-1,0)$. The result in the remaining Sobolev spaces $H^s(\Gamma),\ s\in(0,1]$ follows then from duality arguments.
\end{proof}

\begin{theorem}\label{thm2}
Assume that $Z^2\in\mathbb{C}$ such that $\pm\Im{Z^2}>0$. The operators $\mathcal{A}_{k,\kappa}^2$ defined in equations~\eqref{eq:CFIER2} are invertible with continuous inverses in the spaces $H^s(\Gamma)$ for all $s\in[-1,1]$.
\end{theorem}
\begin{proof}
The fact that the operators $\mathcal{A}_{k,\kappa}^2$  are Fredholm of index 0 in $H^0(\Gamma)$ follows from the same arguments as in Theorem~\ref{thm1}. Indeed, we have 
\begin{eqnarray}
\mathcal{A}_{k,\kappa}^2&=&(I+K_0-2K_0^2) + \mathcal{A}_0^2=2\left(\frac{1}{2}I+K_0\right)\left(I-K_0\right)+ \mathcal{A}_0^2\nonumber\\
\mathcal{A}_0^2&=&2S_\kappa(N_\kappa-N_k)-2(K_\kappa-K_0)K_\kappa-2K_0(K_\kappa-K_0)+S_\kappa Z^2-2S_\kappa K_k^\top Z^2\nonumber\\
& +&(K_k-K_0)+S_k Z^2,\nonumber
\end{eqnarray}
and thus the operator $\mathcal{A}_{k,\kappa}^2$ is a compact perturbation of a Fredholm operator of index 0 in the space $H^0(\Gamma)$. The transpose of the operator $\mathcal{A}_{k,\kappa}^2$ is equal to 
\[
(\mathcal{A}_{k,\kappa}^2)^\top=\frac{1}{2}I-2N_k S_\kappa+Z^2 S_\kappa- 2 Z^2 K_k S_\kappa+K_k^\top+Z^2 S_k.
\]
Let $\psi\in Ker((\mathcal{A}_{k,\kappa}^2)^\top)$ and let us define
\[
w:=SL_k\psi-DL_k[2S_\kappa]\psi,\qquad \mathrm{in}\ \mathbb{R}^2\setminus\Gamma.
\]
We have that
\begin{eqnarray}
\gamma_D^2 w &=& S_\kappa\psi - 2K_k S_\kappa \psi+S_k \psi\nonumber\\
\gamma_N^2 w &=&\frac{1}{2}\psi + K_k^\top\psi - 2N_k S_\kappa\psi\nonumber
\end{eqnarray}
and hence
\[
\gamma_N^2 w + Z^2 \gamma_D^2 w = 0
\]
if we take into account that $\psi\in Ker((\mathcal{A}_{k,\kappa}^2)^\top)$. Now $w$ is a solution of Helmholtz equation in $D_2$ satisfying the impedance boundary condition $\gamma_N^2 w + Z^2 \gamma_D^2 w = 0$. Under the assumption that $\Im{Z^2}\neq 0$ we have that $w$ is identically zero in $D_2$, and hence 
\[
\gamma_D^2 w = 0\qquad \gamma_N^2 w =0.
\]
The last relation immediately implies
\[
\gamma_D^1 w = -2S_\kappa\psi\qquad \gamma_N^1 w = -\psi.
\]
Thus, $w$ is a radiative solution of the Helmholtz equation in $D_1$ that satisfies
\[
\Im \int_\Gamma \overline{\gamma_N^1 w}\ \gamma_D^1 w\ ds = 2\ \Im \int_\Gamma (S_\kappa \psi)\ \overline{\psi}\ ds\geq 0
\]
 which implies that $w=0$ in $D_1$. Consequently, the operator $(\mathcal{A}_{k,\kappa}^2)^\top$ is injective, and thus the operator $\mathcal{A}_{k,\kappa}^2$ is injective as well, which completes the proof in the space $H^0(\Gamma)$. Clearly, the arguments of the proof can be repeated verbatim in the Sobolev spaces $H^s(\Gamma)$ for all $s\in[-1,0)$. The result in the remaining Sobolev spaces $H^s(\Gamma),\ s\in(0,1]$ follows then from duality arguments.
\end{proof}
\begin{remark}
The results in Theorem~\ref{thm1} and Theorem~\ref{thm2} remain valid in the case when $Z^1\in H^1(\Gamma),\ \Im{Z^1}>0$ and $Z^2\in H^1(\Gamma),\ \pm\Im{Z^2}> 0$. Also, in the physically important cases when $Z^1\in L^\infty(\Gamma),\ \Im{Z^1}>0$ and $Z^2\in L^\infty(\Gamma),\ \pm\Im{Z^2}> 0$ (e.g. $Z^j$ are bounded but discontinuous), the CFIER equations~\eqref{eq:CFIER1} and~\eqref{eq:CFIER2} respectively are well posed in the spaces $H^0(\Gamma)$ for impedance data $f^j\in H^0(\Gamma)$.
\end{remark}

\section{Transmission impedance boundary value problems}\label{transmission}

We investigate next regularized formulations for transmission impedance boundary value problems that appear in Domain Decomposition Methods. Domain Decomposition Methods (DDM) are a class of algorithm for the solution of Helmholtz equations that consist of (1) decomposing the computational domain into smaller subdomains, and (2) interconnecting the solutions of subdomain problems by matching impedance conditions on the common interfaces between subdomains~\cite{boubendirDDM}. Fixed point considerations allow to recast the DDM algorithm in terms of the iterative solution of a linear system whose unknown is the global Robin (impedance) data defined on the union of all the subdomain interfaces. The choice of impedance conditions impacts considerably the rate of convergence of the iterative fixed point DDM algorithms. For instance, the use of piecewise constant impedances~\cite{Depres} hinders the fast convergence of DDM algorithms~\cite{boubendirDDM}. A remedy that leads to significant improvements in the rate of convergence of the DDM algorithms consists of the use of transmission impedance boundary conditions--that is on each interface $Z^j$ are suitably chosen (transmission) operators~\cite{boubendirDDM,Gander1,Nataf}. For instance, transmission/impedance operators $Z$ defined as Dirichlet-to-Neumann maps corresponding to adjacent subdomains are advocated as nearly optimal choices as the fixed point DDM iteration would converge in just two iterations~\cite{Nataf}. However, Dirichlet-to-Neumann operators, even when properly defined, are expensive to compute and thus their choice is not computationally advantageous. The common recourse is to use approximations of Dirichlet-to-Neumann operators that are inexpensive to compute and lead to well posed (transmission) impedance boundary value problems. Furthermore, given that Dirichlet-to-Neumann operators are non-local operators, it is easier to construct approximations of those in terms of non-local operators  (e.g. boundary integral operators). For instance, in the case of unbounded subdomains, such a choice is given by $Z^1 =2N_\kappa,\ \Im{\kappa}>0$, whereas in the case of bounded subdomains one could in principle choose $Z^2 =-2N_\kappa,\ \Im{\kappa}>0$. We note that similar operators, e.g. $Z=iN_{i\varepsilon},\ \varepsilon>0$ were used in the context of DDM methods~\cite{Steinbach}. These choices of impedance operators are suitable for boundary integral solvers for the ensuing subdomain problems, in any other contexts (e.g. finite element solvers) localized approximations of Dirichlet-to-Neumann operators are preferable~\cite{boubendirDDM}. 

We show in what follows that our CFIER methodology is applicable to both exterior and interior transmission impedance boundary value problems with the kind of impedance operators discussed above. First, given that~\cite{turc2}
\[
\Im\int_\Gamma N_\kappa\psi\ \overline{\psi}\ ds\geq 0,
\]
the arguments in~\cite{mclean:2000} can be extended to show that equations~\eqref{eq:Ac_i} in $D^1$ with $Z^1 =2N_\kappa,\ \Im{\kappa}>0$, or in $D^2$ with $Z^2 =-2N_\kappa,\ \Im{\kappa}>0$, still have unique solutions $u^1\in C^2(D_1)\cap H^1_{\rm loc}(D_1)$ and $u^2\in C^2(D_2)\cap H^1(D_2)$ respectively. We recast the exterior/interior Helmholtz equations with transmission impedance boundary conditions in the form of CFIER equations~\eqref{eq:CFIER1} and~\eqref{eq:CFIER2} respectively. We establish the following result:

\begin{theorem}\label{thm3}
Assume that $Z^1=2N_\kappa$ such that $\Im{\kappa}>0$. The operators $\mathcal{A}_{k,\kappa}^1$ defined in equations~\eqref{eq:CFIER1} are invertible with continuous inverses in the spaces $H^s(\Gamma)$ for all $s\in[-1,1]$.
\end{theorem}
\begin{proof}
We establish first that the operators $\mathcal{A}_{k,\kappa}^1$  are Fredholm of index 0 in $H^0(\Gamma)$. Using Clader\'on's identities we can recast $\mathcal{A}_{k,\kappa}^1$ into the following form
\begin{eqnarray}
\mathcal{A}_{k,\kappa}^1&=&(2I-6K_0^2-4K_0^3) + \mathcal{A}_0^{1,1}=4\left(\frac{1}{2}I-K_0\right)\left(I+K_0\right)^2+ \mathcal{A}_0^{1,1}\nonumber\\
\mathcal{A}_0^{1,1}&=&2S_\kappa(N_\kappa-N_k)-4(K_\kappa-K_0)K_\kappa-4K_0(K_\kappa-K_0)\nonumber\\
&-&4(S_\kappa-S_0)K_k^\top N_\kappa-4S_0(K_k^\top-K_0^\top)N_\kappa-4S_0 K_0^\top (N_\kappa-N_0)\nonumber\\
& +&(K_0-K_k)-2S_k(N_\kappa-N_k)-2(K_k-K_0)K_k-2K_0(K_k-K_0).\nonumber
\end{eqnarray}
It follows from the results in Theorem~\ref{mapping} that $\mathcal{A}_0^{1,1}:H^0(\Gamma)\to H^1(\Gamma)$ continuously, and thus $\mathcal{A}_0^{1,1}:H^0(\Gamma)\to H^0(\Gamma)$ is compact. Also, the operator 
\[
4\left(\frac{1}{2}I-K_0\right)\left(I+K_0\right)^2
\]
is Fredholm of index 0 in $H^0(\Gamma)$ since (a) the operator $\frac{1}{2}I-K_0$ is Fredholm of index 0 in $H^0(\Gamma)$, (b) the operator $I+K_0$ is invertible in $H^0(\Gamma)$, and (c) the two operators commute. We thus conclude that the operator $\mathcal{A}_{k,\kappa}^1$ is a compact perturbation of a Fredholm operator of index 0 in the space $H^0(\Gamma)$, and hence the operator $\mathcal{A}_{k,\kappa}^1$ is itself a Fredholm operator of index 0 in the same space.

Given the Fredholm property of the operator $\mathcal{A}_{k,\kappa}^1$, its invertibility is equivalent to its injectivity. We show in turn that the transpose of this operator with respect to the real scalar product in $H^0(\Gamma)$ is injective. The latter can be seen to equal
\[
(\mathcal{A}_{k,\kappa}^1)^\top=\frac{1}{2}I-2N_k S_\kappa-2N_\kappa S_\kappa- 4 N_\kappa K_k S_\kappa-K_k^\top-2N_\kappa S_k.
\]
Let $\varphi\in Ker((\mathcal{A}_{k,\kappa}^1)^\top)$ and let us define
\[
v:=SL_k\varphi+DL_k[2S_\kappa]\varphi,\qquad \mathrm{in}\ \mathbb{R}^2\setminus\Gamma.
\]
We have that
\begin{eqnarray}
\gamma_D^1 v &=& S_\kappa\varphi + 2K_k S_\kappa \varphi+S_k \varphi\nonumber\\
\gamma_N^1 v &=&-\frac{1}{2}\varphi + K_k^\top\varphi + 2N_k S_\kappa\varphi\nonumber
\end{eqnarray}
and hence
\[
\gamma_N^1 v + 2N_\kappa \gamma_D^1 v = 0
\]
if we take into account that $\varphi\in Ker((\mathcal{A}_{k,\kappa}^1)^\top)$. Now $v$ is a radiative solution of Helmholtz equation in $D_1$ satisfying the impedance boundary condition $\gamma_N^1 v + 2N_\kappa \gamma_D^1 v = 0$. Under the assumption that $\Im{\kappa}>0$ we have that $v$ is identically zero in $D_1$, and hence 
\[
\gamma_D^1 v = 0\qquad \gamma_N^1 v =0.
\]
The last relation immediately implies
\[
\gamma_D^2 v = -2S_\kappa\varphi\qquad \gamma_N^2 v = \varphi
\]
from which we get by the same arguments as in the proof of Theorem~\ref{thm1} that the operator $(\mathcal{A}_{k,\kappa}^1)^\top$ is injective. Thus, the operator $\mathcal{A}_{k,\kappa}^1$ is injective as well which completes the proof in the space $H^0(\Gamma)$. The proof for the remaining spaces $H^s(\Gamma)$ follows from the same arguments used in the proof of Theorem~\ref{thm1}.
\end{proof}

The arguments in the proofs of Theorem~\ref{thm2} and Theorem~\ref{thm3} imply the following result:
\begin{theorem}\label{thm4}
Assume that $Z^2=-2N_\kappa$ such that $\Im{\kappa}>0$. The operators $\mathcal{A}_{k,\kappa}^2$ defined in equations~\eqref{eq:CFIER2} are invertible with continuous inverses in the spaces $H^s(\Gamma)$ for all $s\in[-1,1]$.
\end{theorem}
\begin{proof}
Since
\begin{eqnarray}
\mathcal{A}_{k,\kappa}^2&=&(2I-6K_0^2+4K_0^3) + \mathcal{A}_0^2=4\left(\frac{1}{2}I+K_0\right)\left(I-K_0\right)^2+ \mathcal{A}_0^2\nonumber\\
\mathcal{A}_0^2&:=&2S_\kappa(N_\kappa-N_k)-4(K_\kappa-K_0)K_\kappa-4K_0(K_\kappa-K_0)\nonumber\\
&+&4(S_\kappa-S_0)K_k^\top N_\kappa+4S_0(K_k^\top-K_0^\top)N_\kappa+4S_0 K_0^\top(N_\kappa-N_0)\nonumber\\
& +&(K_k-K_0)-2S_k(N_\kappa-N_k)-2(K_k-K_0)K_k-2K_0(K_k-K_0),\nonumber
\end{eqnarray}
similar arguments to those used in Theorem~\ref{thm3} deliver the Fredholm property of the operators $\mathcal{A}_{k,\kappa}^2$ in the space $L^2(\Gamma)$. The injectivity of the operators $\mathcal{A}_{k,\kappa}^2$, in turn, can be established exactly as in the proof of Theorem~\ref{thm2}.
\end{proof}
\begin{remark}
  Transmission interior impedance boundary value problems with impedance operators of the form $Z^2=2 N_\kappa$ with  $\Im{\kappa}>0$ can also be shown to be well posed. However, the proof of Theorem~\ref{thm4} does not go through in this case. The reason is that the terms that contain the identity are no longer featured in the operators $\mathcal{A}_{k,\kappa}^2$ and thus the Fredholm argument does not follow from the same considerations. 
  \end{remark}

In the case when the wavenumbers differ in adjacent subdomains---see Figure~\ref{fig:domain}, the DDM matching procedure of transmission impedance boundary conditions in principle calls for approximations of subdomain Dirichlet-to-Neumann operators corresponding to different wavenumbers on each interface. For example, this requirement would lead to a Helmholtz equation in the domain $D_1$ with transmission impedance boundary conditions whose operators $Z^2$ should approximate on the interface between $D_1$ and $D_j$ the restriction to that interface of the Dirichlet-to-Neumann operators for the domains $D_j$ and wavenumbers $k_j$ for $j=2,\ldots,5$. The most natural idea would be to use operators $Z^2$ that are restrictions of the operators $-2N_{k_j+i\varepsilon_j},\ j=2,\ldots,5$ to corresponding subdomain interfaces. This procedure would amount to using local interface impedance operators of the form
\[
Z^2_{1j} = -2 R_{1j}N_{k_j+i\varepsilon_j}E_{1j}:\widetilde{H}^{1/2}(\Gamma_{1j})\to H^{-1/2}(\Gamma_{1j}),\ j=2,\ldots,5
\]
where $E_{1j}:\widetilde{H}^{1/2}(\Gamma_{1j})\to H^{1/2}(\Gamma_1)$ is the extension by zero operator, and $R_{1j}:H^{-1/2}(\Gamma_1)\to H^{-1/2}(\Gamma_{1j})$ is the restriction operator defined by duality
\[
\langle R_{1j}\varphi,\psi\rangle=\langle \varphi,E_{1j}\psi\rangle,\quad \varphi\in H^{-1/2}(\Gamma_1),\ \psi\in\widetilde{H}^{1/2}(\Gamma_{1j}).
\]
In the formulas above we denoted $\Gamma_1:=\partial D_1$, and $\Gamma_{1j}:=\partial D_1\cap \partial D_j$ for $j=2,\ldots,5$. It can be clearly seen from the mapping properties of the operators $Z^2_{1j}, j=2,\ldots,5$ that a simple summation of these would not lead to a global impedance operator defined on $\Gamma_1$ that maps $H^{1/2}(\Gamma_1)$ to $H^{-1/2}(\Gamma_1)$. This shortcoming can be overcome by resorting to impedance operators that blend local impedance operators corresponding to interfaces $\Gamma_{1j}, j=2,\ldots,5$ through partitions of unity:
\begin{equation}\label{gen_impedance_4_domains}
  Z^2_b=-2\sum_{j=2}^5 \chi_j N_{k_j+i\varepsilon_j}\chi_j,\quad \varepsilon_j\geq 0
\end{equation}
where $\chi_j,j=2,\ldots,5$ are cut-off functions such that $\sum_{j=2}^5 \chi_j^2=1$ on $\partial D_1$, $\chi_j\in C_0^\infty(\partial D_1), j=2,\ldots,5$, and $\{\mathbf{x}:\chi_j(\mathbf{x})=1\}\subset \partial D_1\cap\partial D_j$ for $j=2,\ldots,5$. We note that
\[
\Im \int_{\Gamma_1} Z^2_b\psi\ \overline{\psi}\ ds =-2\sum_{j=2}^5 \Im \int_{\Gamma_1} \chi_j  N_{k_j+i\varepsilon_j}\chi_j \psi\ \overline{\psi}\ ds=-2\sum_{j=2}^5 \Im\int_{\Gamma_1} N_{k_j+i\varepsilon_j}\psi_j\ \overline{\psi_j}\ ds<0,\quad \psi\neq 0
\]
where $\psi_j:=\chi_j\ \psi$. These types of operators that use partition of unity blending were originally used in~\cite{Levadoux} to construct coercive approximations of Dirichlet to Neumann operators. It can be shown using ideas from~\cite{Levadoux,BorelLevadouxAlouges} that $Z_b^2+2N_{\kappa}$ is a compact operator from $H^{1/2}(\Gamma_1)$ to $H^{-1/2}(\Gamma_1)$ (and by interpolation from $H^1(\Gamma_1)$ to $L^2(\Gamma_1)$), and thus the results in Theorem~\ref{thm4} can be extended to this new choice of impedance operator.

\begin{figure}
\centering
\includegraphics[height=60mm]{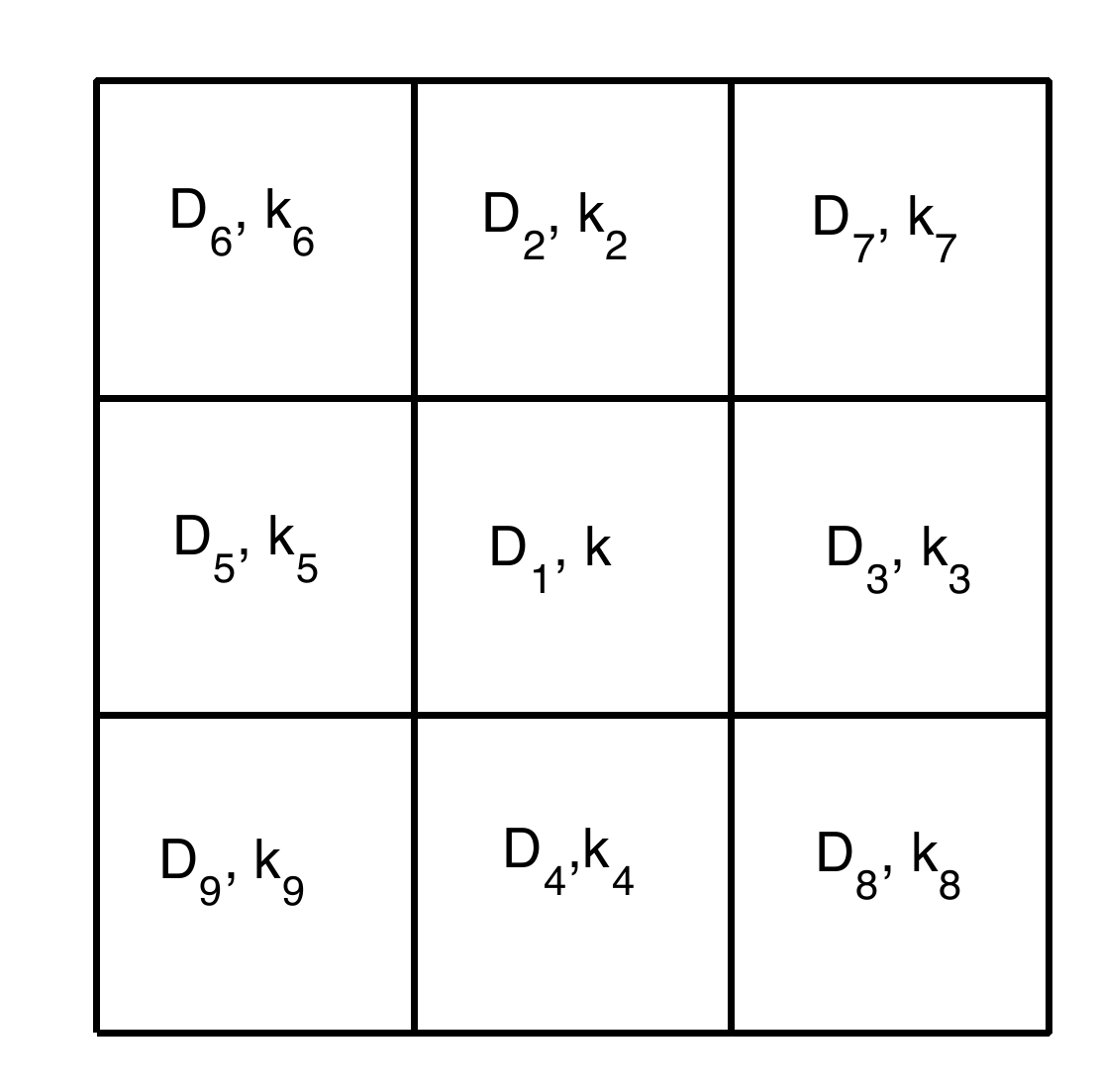}
\caption{Typical DDM configuration.}
\label{fig:domain}
\end{figure}

\section{High-order Nystr\"om methods for the discretization of the CFIER formulations \label{singular_int}}

\parskip 10pt plus2pt minus1pt
\parindent0pt

We present in this section Nystr\"om discretizations of the formulations CFIER~\eqref{eq:CFIER1} and~\eqref{eq:CFIER2} assuming various choices of the impedance $Z^j$. The key components of these discretization are (a) the use of sigmoidal-graded meshes that accumulate points polynomially at corners, (b) the splitting of the kernels of the weighted parametrized operators into smooth and singular components, (c) trigonometric interpolation of the densities of the boundary integral operators, and (d) analytical expressions for the integrals of products of periodic singular and weakly singular kernels and Fourier harmonics. In cases when the impedance $Z^j$ are merely bounded and possibly discontinuous, we reformulate the aforementioned CFIER integral equations in terms of {\em more\ regular} solutions and weighted versions of the boundary integral operators in the Calder\'{o}n's calculus.

We assume that the closed curve $\Gamma$ has corners at $\mathbf{x}_1,\mathbf{x}_2,\ldots,\mathbf{x}_P$ whose apertures measured inside $D_2$ are respectively $\gamma_1,\gamma_2,\ldots,\gamma_P$, and that $\Gamma\setminus\{\mathbf{x}_1,\mathbf{x}_2,\ldots,\mathbf{x}_P\}$ is piecewise analytic. Let  $(x_1(t),x_2(t))$ be a  $2\pi$ periodic parametrization of $\Gamma$  so that each of the ( possibly curved) segments  $[\mathbf{x}_j,\mathbf{x}_{j+1}]$ is mapped by $(x_1(t),x_2(t))$ with $t\in[T_j,T_{j+1}]$. We assume that $x_1(t),x_2(t)$ are continuous and that on each interval $[T_j,T_{j+1}]$  are smooth with
$(x_1'(t))^2+(x_2'(t))^2>0$ (the one-sided derivatives are taken for $t=T_j,T_{j+1}$). Consider  
the sigmoid transform introduced by Kress
\begin{eqnarray}\label{eq:cov_w}
w(s)&=&\frac{T_{j+1}[v(s)]^p+T_j[1-v(s)]^p}{[v(s)]^p+[1-v(s)]^p},\ T_j\leq s\leq T_{j+1},\ 1\leq j\leq P\\
v(s)&=&\left(\frac{1}{p}-\frac{1}{2}\right)\left(\frac{T_{j}+T_{j+1}-2s}{T_{j+1}-T_j}\right)^3+\frac{1}{p}\ \frac{2s-T_j-T_{j+1}}{T_{j+1}-T_j}+\frac{1}{2}\nonumber
\end{eqnarray} 
where $p\geq 2$.  The function $w$ is a smooth, increasing, bijection on each of the intervals $[T_j,T_{j+1}]$ for $1\leq j\leq P$, with $w^{(k)}(T_j)=w^{(k)}(T_{j+1})=0$
for $1\leq k\leq p-1$. We then define the new parametrization
\[
 \mathbf{x}(t)=(x_1(w(t)),x_2(w(t)))
\]
extended by $2\pi-$periodicity, if needed, to any $t\in\mathbb{R}$.

A central issue in Nystr\"om discretizations of the CFIER equations is the regularity of the solutions $\gamma_D^1u$ and $\gamma_D^2u$. In the case when $Z^j\in L^\infty(\Gamma)$ and the impedance data $f_j\in L^2(\Gamma)$ we have already seen that $\gamma_D^ju\in H^1(\Gamma)$ for $j=1,2$. Similarly, in the transmission impedance case, we still have that $\gamma_D^ju\in H^1(\Gamma)$ provided that $f_j\in L^2(\Gamma)$. In all these cases Sobolev embedding results imply that $\gamma_{D}^j u\in C^{0,\beta}(\Gamma)$ for $0<\beta<1$. In the case of piecewise constant impedance $Z^j$ it is more profitable to define weighted Dirichlet traces of solutions of Helmholtz equations 
\[
\gamma_{D}^{j,w}u:=|\mathbf{x}'|\gamma_{D}^j u.
\]
It can be seen that $\gamma_{D}^{j,w}u$ are more regular than $\gamma_{D}^j u$, and their regularity is controlled by the degree $p$ of the sigmoid transform. In addition, the weighted quantities $\gamma_{D}^{j,w}u$ vanish at the corners. We present in what follows parametrized versions of the four boundary integral operators in the Calder\'on calculus. These operators act upon two types of $2\pi$ periodic densities: (1) densities $\varphi\in C^\alpha[0,2\pi]$ where $\alpha$ is large enough which in addition behave as $|t-T_j|^r,\ r>0$ for all $1\leq j\leq P+1$; and (2) densities $\psi\in C^{0,\beta}[0,2\pi], 0<\beta<1$ which are H\"older continuous and periodic.

We start by defining two versions of parametrized single layer operators in the form
\begin{equation*}\label{eq:sl_w}
(S_k^{\mathbf{x},w}\varphi)(t):=
\int_0^{2\pi} G_k(\mathbf x(t) -\mathbf x(\tau))\varphi(\tau)d\tau
\end{equation*}
and
\begin{equation*}\label{eq:sl}
(S_k^{\mathbf{x}}\psi)(t):=
\int_0^{2\pi} G_k(\mathbf x(t) -\mathbf x(\tau))|\mathbf{x}'(\tau)|\psi(\tau)d\tau.
\end{equation*}
We define next two versions of parametrized double layer operators 
\begin{equation*}\label{eq:k_w}
(K_k^{\mathbf{x}}\psi)(t):=\int_0^{2\pi}\frac{\partial G_k(\mathbf x(t)-\mathbf x(\tau))}{\partial\mathbf{n}(\mathbf x(\tau))}|{\bf x}'(\tau)| \psi(\tau)d\tau
\end{equation*}
and
\begin{equation*}\label{eq:k_w1}
(K_k^{\mathbf{x},w}\varphi)(t):=\int_0^{2\pi}\frac{\partial G_k(\mathbf x(t)-\mathbf x(\tau))}{\partial\mathbf{n}(\mathbf x(\tau))}|{\bf x}'(t)| \varphi(\tau)d\tau
\end{equation*}
and two versions of parametrized adjoint double layer operators defined as
\begin{equation*}\label{eq:kt_w}
(K_k^{\mathbf{x},\top,w}\varphi)(t):=\int_0^{2\pi}|{\bf x}'(t)|\frac{\partial G_k(\mathbf x(t)-\mathbf x(\tau))}{\partial\mathbf{n}(\mathbf x(t))} \varphi(\tau)d\tau  
\end{equation*}
and
\begin{equation*}\label{eq:kt}
(K_k^{\mathbf{x},\top}\psi)(t):=\int_0^{2\pi}|{\bf x}'(t)|\frac{\partial G_k(\mathbf x(t)-\mathbf x(\tau))}{\partial\mathbf{n}(\mathbf x(t))} |\mathbf{x}'(\tau)|\psi(\tau)d\tau.  
\end{equation*}
Finally, we defined two versions of parametrized weighted hypersingular operator as
\begin{equation*}\label{eq:n_w}
\begin{split}
(N_k^{\mathbf{x}}\psi)(t):=&
k^{2}\int_0^{2\pi} G_k(\mathbf x(t) -\mathbf x(\tau))
|\mathbf{x}'(t)|\ |\mathbf{x}'(\tau)| (\mathbf{n}(\mathbf x(t))\cdot\mathbf{n}(\mathbf x(\tau)))\psi(\tau)d\tau\\
&+ {\rm PV}
\int_\Gamma |\mathbf{x}'(t)|(\partial_s G_k)(\mathbf x(t) -\mathbf x(\tau)) \psi'(\tau)d\tau
\end{split}
\end{equation*}
and
\begin{equation*}\label{eq:n_w1}
\begin{split}
(N_k^{\mathbf{x},w}\varphi)(t):=&
k^{2}\int_0^{2\pi} G_k(\mathbf x(t) -\mathbf x(\tau))
|\mathbf{x}'(t)|\ (\mathbf{n}(\mathbf x(t))\cdot\mathbf{n}(\mathbf x(\tau)))\varphi(\tau)d\tau\\
&+ {\rm PV}
\int_\Gamma |\mathbf{x}'(t)|(\partial_s G_k)(\mathbf x(t) -\mathbf x(\tau)) \frac{d}{d\tau}\left(\frac{\varphi(\tau)}{|\mathbf{x}'(\tau)|}\right)d\tau
\end{split}
\end{equation*}
We incorporate the parametrized versions of the four boundary integral operators of Calder\'on calculus into parametrized versions of the CFIER formulations considered in this text. First, we use Calder\'{o}n's identities to express the integral operators in the CFIER formulations~\eqref{eq:CFIER1} in the following form that bypasses direct evaluation of hypersingular operators
\begin{equation}\label{eq:CFIER1C}
\mathcal{A}_{k,\kappa}^1= I-2S_\kappa (N_k-N_\kappa)-2K_\kappa^2 -S_\kappa Z^1-2S_\kappa K_k^\top Z^1-K_k-S_k Z^1,\qquad Z^1\in L^\infty(\Gamma)
\end{equation}
and
\begin{eqnarray}\label{eq:CFIER1CN}
\mathcal{A}_{k,\kappa}^1&=& 2I-2(S_\kappa -S_k) (N_k-N_\kappa)-4K_\kappa^2 -2K_k^2\nonumber\\
&-&4S_\kappa K_k^\top(N_\kappa-N_k)-4S_\kappa(N_k-N_\kappa)K_k-4K_\kappa^2K_k,\qquad Z^1=2N_\kappa.
\end{eqnarray}
 Similar considerations apply in the case of the CFIER formulations for the interior impedance problems~\eqref{eq:CFIER2}. Using the parametrized versions of the boundary integral operators described above, we consider both non-weighted and weighted parametrized versions of the CFIER equations. Specifically, we discretize equations~\eqref{eq:CFIER1} using the operators
\begin{eqnarray}\label{eq:CFIER1param}
\mathcal{A}_{k,\kappa}^{\mathbf{x},1}&=& I-2S_\kappa^{\mathbf{x},w} [(N_k^{\mathbf{x}}-N_0^{\mathbf{x}})- (N_\kappa^{\mathbf{x}}-N_0^{\mathbf{x}})]-2(K_\kappa^{\mathbf{x}})^2 -S_\kappa^{\mathbf{x}} Z^1\nonumber\\
&-&2S_\kappa^{\mathbf{x},w} K_k^{\top,\mathbf{x}} Z^1-K_k^{\mathbf{x}}-S_k^{\mathbf{x}} Z^1,
\end{eqnarray}
where $N_0^{\mathbf{x}}$ are the parametrized hypersingular operators for $k=0$, and we solve the parametrized integral equation
\begin{equation}\label{eq:CFIER_pw}
\mathcal{A}_{k,\kappa}^{\mathbf{x},1}\gamma_D^{1}u=-\gamma_N^1 u^{inc}-Z^1\gamma_D^1 u^{inc}.
\end{equation}
We note that the difference operators $N_k^{\mathbf{x}}-N_0^{\mathbf{x}}$ can be written in a simpler form that does not involve differentiation~\cite{dominguez2015well}. We use similar albeit slightly more complicated discrete versions in the case $Z^1=2N_\kappa$. In the case when we use weighted Dirichlet traces as unknowns of the CFIER formulations, the underlying parametrized operators take on the following form:
\begin{eqnarray}\label{eq:CFIER11paramw}
\mathcal{A}_{k,\kappa}^{\mathbf{x},1,1}&=& I-2|\mathbf{x}'|S_\kappa^{\mathbf{x},w} [(N_k^{\mathbf{x},w}-N_0^{\mathbf{x},w})- (N_\kappa^{\mathbf{x},w}-N_0^{\mathbf{x},w})]-2(K_\kappa^{\mathbf{x},w})^2 -|\mathbf{x}'|S_\kappa^{\mathbf{x},w} Z^1\nonumber\\
&-&2|\mathbf{x}'|S_\kappa^{\mathbf{x},w} K_k^{\top,\mathbf{x},w} Z^1-K_k^{\mathbf{x},w}-|\mathbf{x}'|S_k^{\mathbf{x},w} Z^1,
\end{eqnarray}
and we solve the parametrized weighted integral equation
\begin{equation}\label{eq:CFIER_weighted}
\mathcal{A}_{k,\kappa}^{\mathbf{x},1,1}\gamma_D^{1,w}u=-|\mathbf{x}'|\gamma_N^1 u^{inc}-|\mathbf{x}'|Z^1\gamma_D^1 u^{inc}.
\end{equation}
We denote by $\mathcal{A}_{k,\kappa}^{\mathbf{x},2}$ and $\mathcal{A}_{k,\kappa}^{\mathbf{x},2,1}$ the counterparts of the operators $\mathcal{A}_{k,\kappa}^{\mathbf{x},1}$ and $\mathcal{A}_{k,\kappa}^{\mathbf{x},1,1}$ for interior impedance boundary value problems. The parametrized integral operators that feature in equations~\eqref{eq:CFIER1param} and~\eqref{eq:CFIER11paramw} can be expressed in the generic form
\[
(\mathcal{I}\varphi)(t)=\int_0^{2\pi}I(t,\tau)\varphi(\tau)d\tau
\]
where
\[
I(t,\tau)=I_1(t,\tau)\ln\left(4\sin^2\frac{t-\tau}{2}\right)+I_2(t,\tau)
\]
with $I_1(t,\tau)$ and $I_2(t,\tau)\varphi(\tau)$ being regular enough functions that in particular are {\em bounded} for $t=\tau$~\cite{dominguez2015well}. The splitting techniques presented above can be adapted for the evaluation of the operators that involve $\kappa,\ \Im{\kappa}>0$ using additional smooth cutoff function supported in neighborhoods of the target points $t$ according to the procedures introduced in~\cite{turc1}.

In order to derive fully discrete versions of the CFIER equations~\eqref{eq:CFIER_pw} and~\eqref{eq:CFIER_weighted} we use global trigonometric interpolation of the quantities $\gamma_D^1 u$ and $\gamma_D^{1,w}u$. We choose an equi-spaced splitting of the interval $[0,2\pi]$ into $2n$ points so that the meshsize is equal to $h=\pi/n$. We note that since $T_j$ are chosen such that $T_{j+1}-T_j$ are proportional (with the same constant of proportionality) to the lengths of the arcs of $\Gamma$ from $\mathbf{x}_j$ to $\mathbf{x}_{j+1}$ for all $j$, the number of discretization points per subinterval $[T_j,T_{j+1}],\ 1\leq j\leq P$ may differ from each other. We thus consider the equi-spaced collocation points $\{t_0^{(n)}+h/2,t_1^{(n)}+h/2,\ldots,t_{2n-1}^{(n)}+h/2\}$ that exclude corner points and  the interpolation problem with respect to these nodal points
in the space $\mathbb{T}_n$ of trigonometric 
polynomials of the form
$$v(t)=\sum_{m=0}^n a_m\cos{mt}+\sum_{m=1}^{n-1}b_m\sin{mt}$$
is uniquely solvable~\cite{Kress}. We denote by $P_n:C[0,2\pi]\to 
\mathbb{T}_n$ the corresponding trigonometric polynomial interpolation operator . We use the quadrature rules~\cite{KressH}
\begin{eqnarray}\label{eq:quad1}
\int_0^{2\pi}\ln\left(4\sin^2\frac{t-\tau}{2}\right)f(\tau)d\tau&\approx&\int_0^
{2\pi}\ln\left(4\sin^2\frac{t-\tau}{2}\right)(P_nf)(\tau)d\tau\nonumber\\
&=&\sum_{i=0}^{2n-1}R_i^{(n)}(t)f(t_i^{(n)})
\end{eqnarray}
where the expressions $R_j^{(n)}(t)$ are given by
$$R_i^{(n)}(t)=-\frac{2\pi}{n}\sum_{m=1}^{n-1}\frac{1}{m}\cos{m(t-t_i^{(n)})}
-\frac{\pi}{n^2}\cos{n(t-t_i^{(n)})}.$$
We also use the trapezoidal rule
\begin{equation}\label{eq:quad2}
\int_0^{2\pi}f(\tau)d\tau\approx\int_0^{2\pi}(P_nf)(\tau)d\tau=\frac{\pi}{n}
\sum_{i=0}^{2n-1}f(t_i^{(n)}).
\end{equation}
Applying these quadratures rule we obtain fully discrete versions of the parametrized operators in equations~\eqref{eq:CFIER1param} and~\eqref{eq:CFIER11paramw}. We note that the same considerations apply to discretizations of interior impedance boundary value problems.

\subsection{Numerical results}

We present in this section a variety of numerical results that
demonstrate the properties of the CFIER formulations considered in this text. Solutions of the linear systems arising from the Nystr\"om discretizations of the transmission integral equations described in Section~\ref{singular_int} are obtained by means of the fully complex, unrestarted version of the iterative solver GMRES~\cite{SaadSchultz}.  The value of the complex wavenumber $\kappa$ in the CFIER formulations considered was taken to be $\kappa=k+i$ in all the numerical experiments; our extensive numerical experiments suggest that these values of $\kappa$ leads to nearly optimal numbers of GMRES iterations to reach desired (small) GMRES relative residuals. We also present in each table the values of the GMRES relative residual tolerances used in the numerical experiments.
 
We present a variety of numerical experiments concerning the following two Lipschitz geometries: (a) a square centered at the origin whose sides equal to 4, and (b) a L-shape scatterer of sides equal to 4 and indentation equal to 2. We illustrate the performance of our solvers based on the Nystr\"om discretization of the CFIER formulations in two cases of boundary data: (1) point source boundary data for interior problems and (2) plane wave incidence for exterior problems, that is scattering experiments. In case (1) we consider
\[
u^0(\mathbf{x}):=\frac{i}{4}H_0^{(1)}(k|\mathbf{x}-\mathbf{x}_0|),\ \mathbf{x}\in D^2,\ \mathbf{x}_0\in\mathbb{R}^2\setminus \overline{D^2}
\]
and an impedance boundary data constructed as
\[
f^2:=\gamma_N^2u^0+Z^2\gamma_D^2 u^0. 
\]
Clearly, the solution of the interior impedance boundary value problem with data $f^2$ defined above must equal $u^0$ in the domain $\overline{D^2}$. Therefore, in all the numerical experiments that involve interior problems we report the error between the computed boundary values of the solution of the interior impedance boundary value problem with data $f^2$ and the exact boundary values of $u^0$ defined above:
\begin{equation}\label{errorGamma}
  \varepsilon_\Gamma={\rm max}|\gamma_D^2 u^{2,\rm
calc}(\mathbf{x})-\gamma_D^2u^0(\mathbf{x})|
\end{equation}
at the grids points $\mathbf{x}\in\Gamma$ where the numerical solution $\gamma_D^2 u^{2,\rm calc}$ is computed. We note that the latter quantity $\gamma_D^2 u^{2,\rm calc}$ is actually the solution of the discretizations of the CFIER formulations considered in this text. In the case when we use weighted interior formulations of the type~\eqref{eq:CFIER_weighted}, we adjust slightly the definition of the error~\eqref{errorGamma} in the following form
\begin{equation}\label{errorGammaw}
  \varepsilon_\Gamma^w={\rm max}|\gamma_D^{2,w} u^{2,\rm
calc}(\mathbf{x})-\gamma_D^{2,w}u^0(\mathbf{x})|
\end{equation}
given that $\gamma_D^{2,w} u^{2,\rm calc}$ is actually the solution of the weighted CFIER formulations that is being numerically computed.

For every scattering experiment we consider plane-wave incidence $u^{\rm inc}$ and we present maximum far-field errors, that is we choose sufficiently many directions $\hat{\mathbf{x}}=\frac{\mathbf{x}}{|\mathbf{x}|}$ (more precisely 1024 such directions) and for each direction we compute the far-field amplitude $u^{1}_\infty(\hat{\mathbf{x}})$ defined as
\begin{equation}
\label{eq:far_field}
u^{1}(\mathbf{x})=\frac{e^{ik|\mathbf{x}|}}{\sqrt{|\mathbf{x}|}}\left(u^{1}_\infty(\hat{\mathbf{x}})+\mathcal{O}\left(\frac{1}{|\mathbf{x}|}\right)\right),\
|\mathbf{x}|\rightarrow\infty.\\
\end{equation}
The maximum far-field errors were evaluated through comparisons of the
numerical solutions $u_\infty^{1, \rm calc}$ corresponding to either formulation with reference solutions $u_\infty^{1,\rm ref}$ by means of the relation
\begin{equation}
\label{eq:farField_error}
\varepsilon_\infty={\rm max}|u_\infty^{1,\rm
calc}(\hat{\mathbf{x}})-u_\infty^{1,\rm ref}(\hat{\mathbf{x}})|
\end{equation}
The latter solutions $u_\infty^{1,\rm ref}$ were produced using solutions corresponding with refined discretizations based on the CFIER formulations with GMRES residuals of $10^{-12}$ for all geometries. Besides errors appropriately defined in each case, we display the numbers of iterations required by the GMRES solver to reach specified relative residuals.  We used in the numerical experiments discretizations ranging from 6 to 12 discretization points per wavelength, for frequencies $k$ in the medium to the high-frequency range corresponding to scattering problems of sizes ranging from $5$ to $80$ wavelengths. We used both non-weighted and weighted versions of the CFIER formulations which are referred to in the tables by their underlying integral operators. The columns ``Unknowns'' in all Tables display the numbers of unknowns used in each case, which equal to the value $2n$ defined in Section~\ref{singular_int}. We have used sigmoid transforms with a value $p=3$ in all the numerical experiments.  In all of the scattering experiments we considered point source solutions located at $\mathbf{x}_0=(4,4)$ and plane-wave incident fields of direction $d=(0,-1)$. 

We start by presenting the high-order convergence of our Nystr\"om solvers in Table~\ref{table1} for the case of interior impedance boundary value problems with $Z^2=ik$. The loss of accuracy in the solvers based on the weighted formulations can be attributed to larger condition numbers of the matrices associated with the discretization of the operators $\mathcal{A}_{k,\kappa}^{\mathbf{x},2,1}~\eqref{eq:CFIER11paramw}$. In Table~\ref{table2} we present the high-order convergence of our solvers in the case of (exterior) scattering problems with impedance $Z^1=ik$.

 \begin{table}[!htbp] \scriptsize\centering
\begin{tabular}{|c|c|c|c|c|c|c|c|c|c|}
\hline\multicolumn{10}{|c|}{\text{ Interior Helmholtz problem with impedance boundary condition} $Z^2={ik}$} \\\hline\hline
  \text{wavenumber}&  \text{unknowns} & \multicolumn{4}{|c|}{ \text{Square }} & \multicolumn{4}{|c|}{ \text{L-shaped}} \\\hline
  \multirow{2}{*}{$k$} & \multirow{2}{*}{$2n$} &  \multicolumn{2}{|c|}{
$\mathcal{A}_{k,\kappa}^{\mathbf{x},2}~\eqref{eq:CFIER1param}$} &\multicolumn{2}{|c|}{$\mathcal{A}_{k,\kappa}^{\mathbf{x},2,1}~\eqref{eq:CFIER11paramw}$} &\multicolumn{2}{|c|}{$\mathcal{A}_{k,\kappa}^{\mathbf{x},2}~\eqref{eq:CFIER1param}$} &\multicolumn{2}{|c|}{$\mathcal{A}_{k,\kappa}^{\mathbf{x},2,1}~\eqref{eq:CFIER11paramw}$}\\
				&				  &  \text{Iter}& $\epsilon_\Gamma$  & \text{Iter}& $\epsilon_\Gamma^w$ & \text{Iter}& $\epsilon_\Gamma$  & \text{Iter}& $\epsilon_\Gamma^w$\\
\hline \multirow{5}{*}{2}& 32  & 17 & $3.0. \times 10^{-3}$ &18 & $4.8\times 10^{-2}$& 19 &  $5.4\times 10^{-3}$ & 19 & $3.4\times 10^{-2}$ \\
				   &64  & 24  & $6.0 \times 10^{-4}$ & 30 & $1.7 \times 10^{-2}$ & 26 &  $1.6 \times 10^{-3}$ & 28 & $3.1\times 10^{-2}$\\
				   &128  & 25  & $1.0\times 10^{-4}$ & 32  & $7.6\times 10^{-3}$& 25 &  $2.8\times 10^{-4}$ & 30 & $1.9 \times 10^{-2}$ \\
				   & 256 & 25  & $1.7\times 10^{-5}$ & 30 &  $2.0\times 10^{-3}$ & 25 &  $4.7 \times 10^{-5}$ &  30 & $5.9\times 10^{-3}$ \\
				   & 512 & 25  & $2.6\times 10^{-6}$ & 30 & $4.7\times 10^{-4}$&  25 &  $7.3\times 10^{-6}$ & 31 & $1.5\times 10^{-3}$\\
                                   & 1024 & 25  & $3.8\times 10^{-7}$ & 29 & $6.8\times 10^{-5}$ & 25 &  $1.0\times 10^{-6}$ & 31 & $3.5\times 10^{-4}$ \\
\hline\end{tabular}
\caption{High-order convergence of our solvers for the interior impedance boundary value problem using CFIER fromulations with impedance $Z^2=ik$. We present results for the square and the L-shaped scatterers and we consider both non-weighted and weighted version of CFIER. The GMRES tolerance was taken to be $10^{-12}$.\label{table1}}
\end{table}

 \begin{table}[!htbp] \scriptsize\centering
\begin{tabular}{|c|c|c|c|c|c|c|c|c|c|}
\hline\multicolumn{10}{|c|}{\text{Exterior scattering problem with impedance boundary condition} $Z^1={ik}$} \\\hline\hline
  \text{wavenumber}&  \text{unknowns} & \multicolumn{4}{|c|}{ \text{Square }} & \multicolumn{4}{|c|}{ \text{L-shaped}} \\\hline
  \multirow{2}{*}{$k$} & \multirow{2}{*}{$2n$} &  \multicolumn{2}{|c|}{$\mathcal{A}_{k,\kappa}^{\mathbf{x},1}~\eqref{eq:CFIER1param}$} &\multicolumn{2}{|c|}{$\mathcal{A}_{k,\kappa}^{\mathbf{x},1,1}~\eqref{eq:CFIER11paramw}$} &\multicolumn{2}{|c|}{$\mathcal{A}_{k,\kappa}^{\mathbf{x},1}~\eqref{eq:CFIER1param}$} &\multicolumn{2}{|c|}{$\mathcal{A}_{k,\kappa}^{\mathbf{x},1,1}~\eqref{eq:CFIER11paramw}$}\\
				&				  &  \text{Iter}& $\epsilon_\infty$  & \text{Iter}& $\epsilon_\infty$ & \text{Iter}& $\epsilon_\infty$  & \text{Iter}& $\epsilon_\infty$\\
\hline \multirow{5}{*}{2}& 32  & 17 & $4.0. \times 10^{-2}$ &17 & $5.1\times 10^{-2}$& 29 &  $8.0\times 10^{-2}$ & 32 & $8.7\times 10^{-2}$ \\
				   &64  & 21  & $2.5 \times 10^{-3}$ & 23 & $2.6 \times 10^{-3}$ & 29 &  $2.0 \times 10^{-3}$ & 34 & $4.4\times 10^{-3}$\\
				   &128  & 22  & $8.6\times 10^{-5}$ & 21  & $3.0\times 10^{-4}$& 29 &  $1.0\times 10^{-4}$ & 32 & $3.9 \times 10^{-4}$ \\
				   & 256 & 22  & $9.2\times 10^{-6}$ & 21 &  $4.8\times 10^{-5}$ & 29 &  $1.1 \times 10^{-5}$ &  32 & $8.4\times 10^{-5}$ \\
				   & 512 & 21  & $1.1\times 10^{-6}$ & 21 & $7.7\times 10^{-6}$&  28 &  $1.3\times 10^{-6}$ & 29 & $1.7\times 10^{-5}$\\
                                   & 1024 & 21  & $3.1\times 10^{-7}$ & 19 & $1.2\times 10^{-6}$ & 28 &  $1.2\times 10^{-7}$ & 27 & $3.8\times 10^{-6}$ \\
\hline\end{tabular}
\caption{High-order convergence for the exterior scattering problems with impedance $Z^1=ik$ using CFIER formulations. We present Square and L-shaped scatterer and consider both non-weighted and weighted version of CFIER. The GMRES tolerance was taken to be $10^{-12}$.\label{table2}}
 \end{table}

 We present in Table~\ref{table3} the performance of solvers in the high-frequency regime of scattering problems with impedance $Z^1=ik$. Remarkably, the numbers of iterations required to reach a GMRES residual of $10^{-4}$ are small and vary very mildly with increased frequencies. This is also the case for interior impedance boundary problems with $Z^2=-ik$. However, in the case of interior impedance boundary problems with $Z^2=ik$ the situation is quite different as the numbers of iterations grow considerably with the frequency.
 \begin{table}[!htbp] \scriptsize\centering
\begin{tabular}{|c|c|c|c|c|c|}
\hline\multicolumn{6}{|c|}{\text{Exterior scattering problem with impedance boundary condition} $Z^1={ik}$} \\\hline\hline
  \text{wavenumber}&  \text{unknowns} & \multicolumn{2}{|c|}{ \text{Square }} & \multicolumn{2}{|c|}{ \text{L-shaped}} \\\hline
  \multirow{2}{*}{$k$} & \multirow{2}{*}{$2n$} &  \multicolumn{2}{|c|}{$\mathcal{A}_{k,\kappa}^{\mathbf{x},1}~\eqref{eq:CFIER1param}$} &\multicolumn{2}{|c|}{$\mathcal{A}_{k,\kappa}^{\mathbf{x},1}~\eqref{eq:CFIER1param}$} \\
				&				  &  \text{Iter}& $\epsilon_\infty$  & \text{Iter}& $\epsilon_\infty$ \\
  \hline 8 & 192  & 16 & $1.1 \times 10^{-4}$ & 19 &  $1.4\times 10^{-4}$  \\
  \hline 16 & 384  & 17 & $9.3 \times 10^{-5}$ & 19 &  $7.6\times 10^{-5}$ \\
  \hline 32 & 768  & 20 & $1.4 \times 10^{-4}$ &  21 & $1.1\times 10^{-4}$ \\
  \hline 64 & 1536  & 19 & $8.9 \times 10^{-5}$ &  21 & $7.5\times 10^{-5}$ \\
  \hline 128 & 3072  & 22 & $1.2 \times 10^{-4}$ &  24 & $1.1\times 10^{-4}$ \\
\hline\end{tabular}
\caption{Accuracy and numbers of iterations for the solution of exterior scattering problem with impedance $Z^1=ik$ using CFIER formulations. We present Square and L-shaped scatterer and consider both non-weighted and weighted version of CFIER. The GMRES tolerance was taken to be $10^{-4}$. In the case of interior impedance boundary value problems with impedance $Z^2=ik$, the numbers of GMRES iterations needed to reach the same GMRES tolerance are $30, 50, 98, 194, 451$ (square) and $29, 50, 99, 214, 477$ (L-shape) and respectively $12, 14, 16, 19, 22$ (square) and $13, 15, 17, 20, 24$ (L-shape) in the case $Z^2=-ik$ for the same wavenumbers and discretization size leading to comparable levels of accuracy.\label{table3}}
 \end{table}

 In Table~\ref{table4} we present the high-order accuracy of our solvers in the case of interior transmission impedance boundary value problems with $Z^2=-2N_{k+i}$. We continue in Table~\ref{table5} with the high-frequency behavior of our solvers for transmission impedance boundary value problems with $Z^1=2N_{k+i}$. Again, the solvers for exterior problems require very small numbers of iterations for convergence.  

 \begin{table}[!htbp] \scriptsize\centering
\begin{tabular}{|c|c|c|c|c|c|}
\hline\multicolumn{6}{|c|}{\text{Interior problem with impedance boundary condition} $Z^2={-2N_{k+i}}$} \\\hline\hline
  \text{wavenumber}&  \text{unknowns} & \multicolumn{2}{|c|}{ \text{Square }} & \multicolumn{2}{|c|}{ \text{L-shaped}} \\\hline
  \multirow{2}{*}{$k$} & \multirow{2}{*}{$2n$} &  \multicolumn{2}{|c|}{$\mathcal{A}_{k,\kappa}^{\mathbf{x},2}~\eqref{eq:CFIER1param}$} &\multicolumn{2}{|c|}{$\mathcal{A}_{k,\kappa}^{\mathbf{x},2}~\eqref{eq:CFIER1param}$} \\
				&				  &  \text{Iter}& $\epsilon_\Gamma$  & \text{Iter}& $\epsilon_\Gamma$ \\
  \hline \multirow{6}{*}{2} & 32  & 14 & $2.6 \times 10^{-3}$ & 15 &  $5.5\times 10^{-3}$  \\
  & 64  & 14 & $3.0 \times 10^{-4}$ & 15 &  $1.0\times 10^{-3}$ \\
  & 128  & 14 & $5.1 \times 10^{-5}$ &  14 & $1.6\times 10^{-4}$ \\
  & 256  & 14 & $7.8 \times 10^{-6}$ &  14 & $2.5\times 10^{-5}$ \\
  & 512  & 14 & $1.1 \times 10^{-6}$ &  14 & $3.8\times 10^{-6}$ \\
  & 1024  & 14 & $1.6 \times 10^{-7}$ &  14 & $5.6\times 10^{-7}$ \\
\hline\end{tabular}
\caption{High-order convergence of our solvers for the interior Transmission Impedance boundary value problems with $Z^2={-2N_{k+i}}$ using CFIER formulations. We present Square and L-shaped scatterer and considered a GMRES residual equal to $10^{-12}$.\label{table4}}
 \end{table}

 \begin{table}[!htbp] \scriptsize\centering
\begin{tabular}{|c|c|c|c|c|c|}
\hline\multicolumn{6}{|c|}{\text{Exterior scattering problem with impedance boundary condition} $Z^1=2N_{k+i}$} \\\hline\hline
  \text{wavenumber}&  \text{unknowns} & \multicolumn{2}{|c|}{ \text{Square }} & \multicolumn{2}{|c|}{ \text{L-shaped}} \\\hline
  \multirow{2}{*}{$k$} & \multirow{2}{*}{$2n$} &  \multicolumn{2}{|c|}{$\mathcal{A}_{k,\kappa}^{\mathbf{x},1}~\eqref{eq:CFIER1param}$} &\multicolumn{2}{|c|}{$\mathcal{A}_{k,\kappa}^{\mathbf{x},1}~\eqref{eq:CFIER1param}$} \\
				&				  &  \text{Iter}& $\epsilon_\infty$  & \text{Iter}& $\epsilon_\infty$ \\
  \hline 8 & 192  & 8 & $6.1 \times 10^{-4}$ & 9 &  $5.8\times 10^{-4}$  \\
  \hline 16 & 384  & 8 & $2.8 \times 10^{-4}$ & 9 &  $4.0\times 10^{-4}$ \\
  \hline 32 & 768  & 8 & $2.6 \times 10^{-4}$ &  9 & $3.8\times 10^{-4}$ \\
  \hline 64 & 1536  & 6 & $2.9 \times 10^{-4}$ &  9 & $4.7\times 10^{-4}$ \\
  \hline 128 & 3072  & 6 & $2.8 \times 10^{-4}$ &  9 & $4.1\times 10^{-4}$ \\
\hline\end{tabular}
\caption{Accuracy and numbers of iterations for the solution of exterior scattering problem with transmission impedance operator $Z^1=2N_{k+i}$ using CFIER formulations. We present Square and L-shaped scatterer and consider both non-weighted and weighted version of CFIER. The GMRES tolerance was taken to be $10^{-4}$. In the case of interior impedance boundary value problems with impedance operators $Z^2=-2N_{k+i}$, the numbers of GMRES iterations needed to reach the same GMRES tolerance are $7, 7, 7, 7, 7$ (square) and respectively $8, 7, 8, 8, 8$ (L-shape) for the same wavenumbers and discretization size leading to comparable levels of accuracy.\label{table5}}
 \end{table}

We continue in Table~\ref{table6} with  scattering experiments for the physically important case of piecewise constant impedance. In this case, given that the impedance data $f^1$ is discontinuous, we employ the weighted version of CFIER formulation to obtain numerical solutions. Finally, we present in Table~\ref{table7} results for the case of interior problems with blended transmission impedance operators $Z^2_b$ defined in equations~\eqref{gen_impedance_4_domains}. Given that the main motivation for these problems comes from DDM, we focus on the case of square subdomains. As it can be seen from the results in Table~\ref{table7}, the efficiency of the CFIER formulations deteriorates with the growth of the size of the central subdomain $D_1$ in Figure~\ref{fig:domain}. A possible remedy to this situation is to further subdivide the subdomain $D_1$ into smaller subdomains.   
 \begin{table}[!htbp] \scriptsize\centering
\begin{tabular}{|c|c|c|c|c|c|}
\hline\multicolumn{6}{|c|}{\text{Exterior scattering problem with piecewise constant impedance boundary condition}} \\\hline\hline
  \text{wavenumber}&  \text{unknowns} & \multicolumn{2}{|c|}{ \text{Square }} & \multicolumn{2}{|c|}{ \text{L-shaped}} \\\hline
  \multirow{2}{*}{$k$} & \multirow{2}{*}{$2n$} &  \multicolumn{2}{|c|}{$\mathcal{A}_{k,\kappa}^{\mathbf{x},1,1}~\eqref{eq:CFIER11paramw}$} &\multicolumn{2}{|c|}{$\mathcal{A}_{k,\kappa}^{\mathbf{x},1,1}~\eqref{eq:CFIER11paramw}$} \\
				&				  &  \text{Iter}& $\epsilon_\infty$  & \text{Iter}& $\epsilon_\infty$ \\
  \hline 8 & 192  & 22 & $2.4 \times 10^{-4}$ & 23 &  $3.0\times 10^{-4}$  \\
  \hline 16 & 384  & 26 & $1.3 \times 10^{-4}$ & 27 &  $1.2\times 10^{-4}$ \\
  \hline 32 & 768  & 30 & $1.6 \times 10^{-4}$ & 32 & $1.3\times 10^{-4}$ \\
  \hline 64 & 1536  & 35 & $2.1 \times 10^{-4}$ & 37  & $1.6\times 10^{-4}$ \\
  \hline 128 & 3072  & 42 & $1.5 \times 10^{-4}$ & 42  & $2.1\times 10^{-4}$ \\
\hline\end{tabular}
\caption{Results for the exterior scattering problem using weighted CFIER formulations in the case of piecewise constant impedance boundary condition with impedance operator $Z^1=i\alpha_j k$. The coefficients $\alpha_j$ were chosen so that $\alpha_j = j-1$ along the $j$-th side of the scatterer. We present Square and L-shaped scatterer. The GMRES residual was taken to be equal to $10^{-4}$.\label{table6}}
 \end{table}

 \begin{table}[!htbp] \scriptsize\centering
\begin{tabular}{|c|c|c|c|}
\hline\multicolumn{4}{|c|}{\text{Interior problem with blended transmission mpedance boundary conditions $Z^2_b$~\eqref{gen_impedance_4_domains}}} \\\hline\hline
  \text{wavenumber}&  \text{unknowns} & \multicolumn{2}{|c|}{ \text{Square }}  \\\hline
  \multirow{2}{*}{$k$} & \multirow{2}{*}{$2n$} &  \multicolumn{2}{|c|}{$\mathcal{A}_{k,\kappa}^{\mathbf{x},2}~\eqref{eq:CFIER1param}$}  \\
				&				  &  \text{Iter}& $\epsilon_\Gamma$  \\
  \hline 4 & 64  & 15 & $2.5 \times 10^{-4}$  \\
  \hline 8 & 128  & 29 & $4.3 \times 10^{-4}$  \\
  \hline 16 & 256  & 72 & $6.0 \times 10^{-4}$ \\
  \hline 32 & 512  & 107 & $3.0 \times 10^{-4}$ \\
\hline\end{tabular}
\caption{Results for the interior problem with blended transmission impedance boundary conditions $Z^2_b$~\eqref{gen_impedance_4_domains}. The complexified wavenumbers in the adjacent domains that enter the definition of the operator $Z^2_b$ were taken to be equal to $1+i, 2+i, 3+i, 4+i$. The GMRES residual was taken to be equal to $10^{-4}$.\label{table7}}
 \end{table}

\section{Conclusions}

In this work we have presented high-order Nytr\"om discretizations based on polynomially graded meshes for regularized boundary integral formulations for Helmholtz impedance boundary value problems in domains with corners. We have rigorously proven  the well-posedness of the regularized formulations and we have shown that the Nyst\"om discretizations of these formulations lead to efficient and very accurate solvers of impedance boundary value problems. The numerical analysis of these schemes will be subject of future investigation. 

\section*{Acknowledgments}
Catalin Turc gratefully acknowledge support from NSF through contract DMS-1312169. Yassine Boubendir gratefully acknowledge support from NSFthrough contract DMS-1319720.

\bibliography{biblio}

\end{document}